\documentclass[12pt]{amsart}

\usepackage{lmodern}
\usepackage[english]{babel}
\usepackage{mathtools}
\usepackage[hidelinks]{hyperref}

\numberwithin{equation}{section} 
\theoremstyle{plain} 
\newtheorem{theorem}{Theorem}[section]
\newtheorem{lemma}[theorem]{Lemma}
\newtheorem{claim}[theorem]{Claim}

\newtheorem{corollary}[theorem]{Corollary}
\newtheorem{remark}[theorem]{Remark}
\newtheorem{example}[theorem]{Example}

\theoremstyle{definition}

\usepackage{xcolor}

\title[On extreme points in shift-invariant spaces]{Extreme and exposed points of shift-invariant spaces generated by Gaussian kernel and hyperbolic secant}

 \author[Hagen]{Markus Valås Hagen}
 \address{Department of Mathematical Sciences, Norwegian University of Science and Technology (NTNU), 7491 Trondheim, Norway}
 \email{markus.v.hagen@ntnu.no} 
 
 \author[Ulanovskii]{Alexander Ulanovskii}
 \address{Department of Mathematics and Physics, University of Stavanger, 4036 Stavanger, Norway}
 \email{alexander.ulanovskii@uis.no} 

\author[Zelent]{Denis Zelent}
\address{Department of Mathematical Sciences, Norwegian University of Science and Technology (NTNU), 7491 Trondheim, Norway} 
 \email{denis.zelent@ntnu.no} 
 
\author[Zlotnikov]{Ilya Zlotnikov}
\address{Department of Mathematical Sciences, Norwegian University of Science and Technology (NTNU), 7491 Trondheim, Norway} 
\email{ilia.k.zlotnikov@ntnu.no}

\newcommand{\R}{{\mathbb R}}  \newcommand{\Z}{{\mathbb Z}} \newcommand{\N}{{\mathbb N}}

\newcommand{\Cc}{{\mathbb C}}

\newcommand{\ext}{{\rm Ext}}

\newcommand{\expo}{\,{\rm Exp}}
\def\res{\mathop{{\rm Res}}\limits}

\begin{document}

\thanks{Hagen, Zelent, and Zlotnikov were supported by Grant 334466 of the Research Council of Norway.}

\begin{abstract}
We characterize the extreme and exposed points of the unit ball (with respect to the $L^1$-norm) in the shift-invariant space generated by the Gaussian function, as well as in the quasi shift-invariant space generated by the hyperbolic secant.
\end{abstract}

\maketitle

\section{Introduction and Main Results}

\subsection{Extreme and exposed points}

Let $X$ be a Banach space and let $C\subset X$ be a closed convex set. An element $f \in C$ is called an extreme point of $C$ if it does not lie in the interior of any nontrivial line segment contained in $C$; that is, there do not exist distinct points $f_1, f_2 \in C$ and a scalar $0 < \alpha < 1$ such that $f = \alpha f_1 + (1-\alpha) f_2$. Among the extreme points, an element $f \in C$ is said to be an exposed point of $C$ if there exists a continuous linear functional $\varphi \in X^\ast$ such that $$\varphi(f) > \varphi(g) \quad \text{for all } g \in C \setminus \{f\}.$$
Equivalently, $\varphi$ attains its  maximum over $C$ uniquely at $f$.

Let
$$B(X)= \{ f \in X : \|f\| \le 1 \}, \qquad S(X) = \{ f \in X : \|f\| = 1 \}$$ denote the unit ball and the unit sphere of $X$.
We denote the sets of extreme and exposed points of $B(X)$ by
$\ext(X)$ and  $\expo(X),$
respectively. Note that $\expo(X)\subset \ext(X)\subset S(X)$.

The sets $\ext(X)$ and $\expo(X)$ play a fundamental role in the geometry of the underlying space through the classical Krein–Milman theorem. They arise naturally in a wide range of problems, including the characterization of linear isometries of function spaces  \cite{MR121646} and applications in prediction theory \cite{MR705750}.
These sets have been explored in numerous function spaces, with analytic function spaces receiving particularly sustained attention. The most significant and interesting cases arise for function spaces equipped with the 
$L^p$-norm when $p=1$ or $p=\infty$ (see also Remark \ref{rm} below). In these settings, the extreme and exposed points of the unit ball have been extensively investigated, especially for the classical Hardy spaces and their variants \cite{MR283557, MR98981, MR4394687, MR348467}, various polynomial spaces \cite{MR1783617, MR2024728, MR4205713, MR2403441},  and certain Paley–Wiener spaces \cite{MR1783617, MR4574376}. For a broader perspective, we refer the reader to the survey \cite{MR4582781}, which also highlights several interesting open problems.

This paper presents, to the best of our knowledge, the first study of the extreme and exposed points of the unit ball in shift-invariant and quasi shift-invariant function spaces. We focus on the classical generators—the Gaussian kernel and the hyperbolic secant—so that the resulting spaces consist of entire or meromorphic functions on the complex plane $\mathbb C$.
These spaces also play an important role in modern sampling theory and in the theory of Gabor frames. From the perspective of these theories, they exhibit very similar properties due to the relation between the Zak transforms of the Gaussian and hyperbolic secant functions, see \cite{MR4047939, MR1884237, MR4782146}. However, as we demonstrate below, their properties differ when examined from the standpoint of Banach space geometry.

\subsection{Shift-invariant spaces and quasi shift-invariant spaces}

Let $1 \leq p \leq \infty$, and let $\Gamma\subset\R$ be a separated set, i.e.,
$$
\inf\limits_{\gamma, \gamma' \in \Gamma,\gamma\ne\gamma'} |\gamma - \gamma'| >0.
$$

The quasi shift-invariant space $V_{\Gamma}^p(g)$ generated by a function $g$ consists of the functions 
$$
f(x) = \sum\limits_{\gamma \in \Gamma} c_{\gamma} g(x-\gamma), \quad  \{c_\gamma\} \in l^p(\Gamma).
$$
In what follows, we denote by $c_\gamma(f)$  the coefficient $c_\gamma$  in the representation above.

Following tradition, we assume throughout that the set of $\Gamma$‑shifts is relatively dense; that is, there exists an $R>0$ such that every interval of length $R$ contains at least one point of $\Gamma$. This assumption is not essential for the subsequent analysis.
 
In the particular case $\Gamma = \Z$, the space $V^p(g):=V^p_{\Z}(g)$ is called shift-invariant.

The space  $V_{\Gamma}^p(g)$ is well-defined whenever $g$ belongs to the Wiener amalgam space $W$, i.e.,
\begin{equation}\label{eq:wiener}
\sum\limits_{k\in\Z} \|g\|_{L^{\infty}[k,k+1]}< \infty,
\end{equation}
since  then $\|f\|_{L^p(\R)}$ is finite for every $f \in V^p_{\Gamma}(g)$, see e.g. \cite[Lemma~8.1]{MR4865221}. 

As usual, we equip the spaces $V^p_{\Gamma}$ with the standard $L^p$-norm.
The quasi shift-invariant spaces considered below satisfy the following condition: 
for every $p \in [1, \infty]$ there exist positive constants $K_1, K_2$ such that
\begin{equation}\label{eq:norm_equiv}
    K_1 \|c\|_{l^p} \le \|f\|_{L^p(\R)} \le K_2 \|c\|_{l^p},\quad \text{ for every $f \in V^p_{\Gamma}(g).$}
\end{equation}
This property shows that  $V^p_{\Gamma}(g)$ is a closed subspace of $L^p(\R)$.

Note that when $\Gamma = \Z$ and $f\in W$, the norm equivalence~\eqref{eq:norm_equiv} holds if and only
if the Fourier transform $\hat{g}$ does not completely vanish on any arithmetic progression of step $1$, 
see \cite{Jia1991}. For the case of general separated sets $\Gamma$, we refer the reader to \cite{MR4950401}. 

Like the Paley--Wiener spaces, (quasi) shift-invariant spaces play a fundamental role in approximation theory. In particular, the sampling and interpolation properties of the spaces $V^p_{\Gamma}(g)$ are closely connected to the frame properties of the Gabor systems generated by $g$, see, e.g., \cite{MR4456797, grs, MR4047939, MR3053565, MR4782146, MR4865221}. Important examples arise when $g$ is the Gaussian kernel or the hyperbolic secant:
\begin{equation}\label{eq:gauss_secant_def}
    G_a(x):= e^{-ax^2}, \quad H_a(x) = \frac{1}{e^{ax}+{e^{-ax}}}, 
\end{equation}
where $a > 0.$

It is well-known that the spaces $V^p(G_a)$ and $V^p_{\Gamma}(H_a)$ satisfy~\eqref{eq:norm_equiv}, see e.g. \cite{MR3053565}, \cite{MR4865221}, \cite{MR4950401}.
It is easy to check that every element of $V^p(G_a)$ admits analytic extension to an entire function, while every element of $V^p_\Gamma(H_a)$ admits analytic extension to a meromorphic function.

\begin{remark}\label{rm} For the spaces  $V^p(G_a)$ and $V^p_\Gamma(H_a)$ with $1<p<\infty$, the set of extreme points of the unit ball coincides with the unit sphere. This follows directly from the classical equality case in Minkowski’s inequality: $\|f+g\|_{L^p(\R)}=\|f\|_{L^p(\R)}+ \|g\|_{L^p(\R)}$ iff $f=ag$ for some $a\geq0$.  \end{remark}

The present paper examines the geometry of the unit balls of these spaces in the case $p=1.$

\subsection{Main results}

\begin{theorem}\label{th:gauss_ext}{\rm (Gaussian kernel)}
{\rm (i)} The set $\ext(V^1(G_a))$ consists of those functions $f\in V^1(G_a)$ satisfying  $\|f\|_{1} = 1$ and such that 
\begin{equation}\label{s_zeros}
   \text{There is no  point $\lambda\in \mathbb C$ with $0<{\rm Im\,}\lambda<\pi/a$ for which } f(\lambda) = f(\bar \lambda) = 0.\end{equation}

{\rm   (ii)}      The  set $\expo(V^1(G_a))$ consists of those functions $f\in \ext(V^1(G_a))$  for which the following two conditions hold: 
    \begin{equation}\label{eq:exposed_cond}
     \text{There is no point }   \lambda \in \R \quad \text{such that} \quad f(\lambda) = f'(\lambda) = 0;
    \end{equation} 
  
\begin{equation}\label{ex}
\int_\R e^{2ax}|f(x)|\,dx =\int_\R e^{-2ax}|f(x)|\,dx =\infty.
\end{equation}
\end{theorem}

\begin{remark}
Similar conditions to \eqref{s_zeros} and \eqref{eq:exposed_cond} also arise in the setting of polynomials and Paley–Wiener spaces{\rm;} see e.g., \cite[Theorems~1 and~3]{MR1783617}.
\end{remark}

\begin{remark} {\rm (See Lemma \ref{lcl})} Let $f\in V^1(G_a).$ Condition \eqref{ex} is equivalent to requiring that neither sequence $c_{n}(f)e^{2an}$ nor $c_{n}{(f)}e^{-2an}$ belongs to $l^1(\Z)$, i.e.
        $$\sum_{n\in \Z}e^{2an}|c_n(f)|=\sum_{n \in \Z}e^{-2an}|c_n(f)|=\infty.$$
\end{remark}
 
For the hyperbolic secant, we obtain the description of $\ext(V^1_{\Gamma}(H_a))$ and $\expo(V^1_{\Gamma}(H_a))$  for general separated sets of translates $\Gamma.$ 

\begin{theorem}\label{th:hyp_ext}{\rm (Hyperbolic secant generator)} Let $\Gamma\subset\R$ be a separated relatively dense set.

 {\rm (i)}  The set $\ext(V_\Gamma^1(H_a))$ consists  of those functions $f\in V_\Gamma^1(H_a)$ satisfying $\|f\|_{1} = 1$, condition \eqref{s_zeros} and the condition
\begin{equation}\label{eq:c_non_zero}
    c_\gamma(f) \neq 0, \quad \text{for all \,} \gamma \in \Gamma.
\end{equation}

{\rm (ii)} The set  $\expo(V^1_{\Gamma}(H_a))$ consists of those functions  $f\in \ext(V^1_{\Gamma}(H_a))$  satisfying conditions \eqref{eq:exposed_cond} and \eqref{ex}.
\end{theorem}

\begin{remark}\label{rem} {\rm (See Corollary \ref{c43})} For $f\in V^1_\Gamma(H_a),$ the condition $e^{2ax}f(x)\in L^1(\R)$ is equivalent to the two conditions
  $$\sum_{\gamma\in\Gamma}e^{a\gamma}c_\gamma(f)=0,\quad \sum_{\gamma\in\Gamma}e^{2a\gamma}|c_\gamma(f)|<\infty.$$
  
 Similarly,  $e^{-2ax}f(x)\in L^1(\R)$ if and only if
  $$\sum_{\gamma\in\Gamma}e^{-a\gamma}c_\gamma(f)=0, \quad \sum_{\gamma\in\Gamma}e^{-2a\gamma}|c_\gamma(f)|<\infty.$$
  \end{remark}

We illustrate the above results with a simple example.

\begin{example}
    Consider the function $$f_\sigma(x) = C(e^{-(x-1)^2} - \sigma e^{-(x+1)^2}) = Ce^{-x^2-1}(e^{2x}-\sigma e^{-2x})\in V^1(G_1),\quad \sigma\in\Cc,$$ where $C$ is chosen so that $\|f_{\sigma}\|_{L^1} =1$. Observe that $f_\sigma \in \ext(V^1(G_1))$ precisely when  $\sigma \notin \R.$
    Moreover, $f_\sigma$ is not an exposed point of $B(V^1(G_1))$, since the function $f_\sigma(x) e^{2x} \in L^1(\R).$
    In contrast,  any finite linear combination of translates of the hyperbolic secant clearly fails to lie in $\ext(V^1_{\Gamma}(H_a))$, in view of \eqref{eq:c_non_zero}. 
\end{example}

The proofs of Theorems \ref{th:gauss_ext} and \ref{th:hyp_ext} are complex‑analytic and rely on a deep result of Hayman  \cite{MR111839} concerning the relationship between the growth and decay of entire functions of slow growth. This method is effective because of a key periodicity phenomenon: elements of the quasi shift‑invariant spaces generated by the hyperbolic secant $H_a$ are $2\pi i/a$-periodic, and an analogous property holds for the shift‑invariant space generated by the Gaussian $G_a$ once the functions are multiplied by $\exp(ax^2).$
 We note that this periodicity was already exploited in \cite{MR4047939} to analyze sampling sets for shift‑invariant spaces generated by Gaussian and hyperbolic secant.
 
The paper is organized as follows. Sec. 2 recalls  standard criteria for extreme and exposed points and presents an auxiliary result from the theory of entire functions that will be essential in the sequel. In Sec.~3 we study shift-invariant spaces generated by the Gaussian kernel and prove Theorem~\ref{th:gauss_ext}. The proof of Theorem~\ref{th:hyp_ext} is presented in Sec.~4. 

\section{Preliminaries}

\subsection{Extreme and exposed points criteria}

Assume that $g$  is one of the generators specified in~\eqref{eq:gauss_secant_def}.
Our analysis of the geometry of the unit ball of the (quasi) shift-invariant spaces $V^1_\Gamma(g)$
 will rely on the following two standard lemmas.
\begin{lemma}[Extreme point criterion]\label{h_lemma}
Suppose that $f \in V^1_{\Gamma}(g)$ and $\|f\|_{1} = 1$.
Then $f\in \ext(V^1_{\Gamma}(g))$ if and only if there is no real, non-constant, bounded  function $\tau$ on $\R$ satisfying $f \tau \in V^1_{\Gamma}(g)$.
\end{lemma}
\begin{lemma}[Exposed point criterion]\label{exp_lemma}
A function  $f\in{\rm \expo}(V^1_{\Gamma}(g))$ if and only if  $f\in{\rm \ext}(V^1_{\Gamma}(g))$ and there is no real, non-negative,  non-constant function $h$ on $\R$ satisfying $fh\in V^1_{\Gamma}(g)$.
\end{lemma}

These lemmas are well-known in the settings of the Hardy space $H^1$ and the Paley-Wiener spaces (see \cite{MR1783617, MR410387}). The same arguments apply, without essential modification, to the (quasi) shift-invariant spaces $V^1_{\Gamma}(g)$ generated by Gaussian kernel or hyperbolic secant.

\subsection{A corollary of Hayman's theorem}
Given an entire function $h$, set
\begin{equation}\label{eq:m_M_def}
M(r,h)=\sup_{|z|\leq r}|h(z)|,\ \ m(r,h)=\inf_{|z|=r}|h(z)|.  
\end{equation}

If $h(z)=a_0+a_1z+...+a_nz^n$ is a polynomial, then clearly, for every positive $\varepsilon$ there exists $r_\varepsilon>0$ such that $m(r,h)>(1-\varepsilon)M(r,h)$ for all $r>r_\varepsilon.$ This is no longer true for transcendental entire functions.  However, classical results—most notably the 
$\cos\pi\lambda$-theorem — show that for functions of “slow” growth, the quantity 
$m(r,h)$ cannot be too small relative to 
$M(r,h)$. 

 The following lemma is an immediate corollary of Hayman's classical result (see Theorem 1 and Corollary 1 in  \cite{MR111839}):
\begin{lemma}\label{hayman}
    Assume $\psi$ is an entire function satisfying
    \begin{equation}\label{mrh}
        \log M(r,\psi)\leq C\log^2 r, \ r\geq 2, \quad \text{for some constant } C.
    \end{equation}
    Then, 
\begin{enumerate}
    \item[$(i)$]     
    for almost every $\theta\in [0,2\pi)$ there exists $r_\theta>0$ such that
\begin{equation}\label{h1}
\log|\psi(re^{i\theta})|\geq\frac{\log M(r,\psi)}{2},\quad r\geq r_\theta;
\end{equation}
    \item[$(ii)$] there is a set $E\subset (1,\infty)$ satisfying $\int_E (1/t)\,dt<\infty$ and such that \begin{equation}\label{mrh1}
    \log m(r,\psi)\geq \frac{\log M(r,\psi)}{2},\quad r\in (1,\infty)\setminus E.
\end{equation}
\end{enumerate}
\end{lemma}

\begin{lemma}
    Let $\varphi$ be an entire $\pi i/a$-periodic function. Then 
    $$\varphi(z)=\sum_{n\in\Z}a_ne^{2az},\quad z\in\mathbb C,$$where the series converges uniformly in any strip $-\infty<A\leq\,${\rm Re}$\,z\leq B<\infty$.
\end{lemma}

Indeed, set $w:=e^{2za}$, so that $z=\log w/(2a)$ and ${\rm Re \,}z=\log |w|/(2a)$. Since $\varphi$ is entire and $\pi i /a$-periodic, it is clear that the function $\psi(w):=\varphi(\frac{\log w}{2a})$ is analytic in $\mathbb C\setminus\{0\}$. Hence,
 $$\psi(w)=\sum_{k=-\infty}^\infty a_k w^k,$$ and the series converges uniformly in any ring $0<r<|w|<R<\infty,$ which proves the lemma.      

\begin{lemma}\label{c24}
    Assume that $\varphi=\varphi_1/\varphi_2$ is an entire function where  $\varphi_1$ and $\varphi_2$ are  entire $\pi i/a$-periodic functions satisfying the estimate \begin{equation}\label{eest}
    |\varphi_j(x+iy)|\leq Ce^{K x^2}, \ j=1,2, \quad \text{for all } x,y\in\R \text{ and  some } C,K>0.\end{equation}
    Then  $\varphi$ also satisfies \eqref{eest} with some $C,K>0$.
    \end{lemma}

Although this lemma is a particular case of \cite[Theorem 1 in Lecture 2]{MR1400006}, we present an alternative proof based on Lemma \ref{hayman}, which we will refer to several times later.

\begin{proof}
 Consider the functions 
 $$\psi_j(w):=\varphi_j\left(\frac{\log w}{2a}\right)=\sum_{k=-\infty}^\infty a_k^{(j)} w^k,\ j=1,2.$$  

 Write $$\psi_j(w)=\sum_{k\geq 0}a_k^{(j)} w^k+\sum_{k<0}a_k^{(j)}w^k=:\psi_j^+(w)+\psi_j^-(w).$$ 
 Clearly,  both functions $\psi_j^-(w),j=1,2,$  are bounded for $|w|> 1.$ Hence, by \eqref{eest}, $\psi_j^+$ satisfy \eqref{mrh}, $j=1,2$. 

  Assume  $\psi_2^+ \equiv const$, so that  $n_0:=\max\{k\in\mathbb Z: a_k^{(2)}\ne0\}\leq0$. Then for all $w$ with sufficiently large absolute value, 
 $$|\psi_2(w)| \geq |w|^{n_0}\left(|a^{(2)}_{n_0}| - \left|\sum\limits_{n < n_0} a^{(2)}_n w^{n-n_0} \right|\right) \geq C |w|^{n_0}.$$ 

  Assume now that $\psi_2^+$ is not a constant function.  
By Lemma~\ref{hayman} (ii), there is a sequence $r_k\to\infty$ such that
$$
\log m(r_k,\psi^+_2) > \frac{1}{2}\log M(r_k, \psi^{+}_2)\to\infty,\quad k\to\infty,$$
where $m$ and $M$ are defined in~\eqref{eq:m_M_def}. Hence, $|\psi_2(w)|>C, |w|=r_k,k\in\N.$

The above estimates show  that there is a sequence $r_k\to\infty$ such that \begin{equation}\label{ppsi}
    |\psi_2(w)|\geq C|w|^{n_0}\quad \text{for some } n_0\leq 0 \, \text{ and } \, |w|=r_k\to\infty.
\end{equation}Hence,
\begin{equation}\label{pvsi}
    |\varphi_2(z)|\geq Ce^{2an_0x}\quad \text{for some } n_0\leq0 \, \text{ and } \, x=x_k\to\infty.
\end{equation}

Using \eqref{pvsi} and estimate  \eqref{eest} for $\varphi_1$, we conclude that  $\varphi(z)$ satisfies \eqref{eest}  for all $z = x + i y$, where $x=x_k:=(\log r_k)/(2a) \to+\infty$.

A similar approach shows that $\varphi(z)$ satisfies \eqref{eest} for all $z=x+iy, $ where $x=\tilde x_k\to-\infty.$

We now apply Lemma \ref{hayman} (i) to functions $\psi^{\pm}_1$ to find   numbers $ \theta\in [0,2\pi)$ and $R=R(\theta)>0$ such that both functions $\psi^{+}_1(re^{i\theta})$ and
$\psi^{-}_1((1/r)e^{-i\theta})$
satisfy \eqref{h1} for $r>R$. Similarly to above, we see that $\varphi(z)$ satisfies \eqref{eest}  for all $z=x+iy_0, y_0:=\theta/(2a)$.
Since $\varphi$ is entire and $\pi i/a$-periodic, we have the same estimate for every $y = y_0+\pi  k/a, k\in\Z$.
Therefore, we may find $y_1\leq0$ and $y_2\geq \pi/a$ and a large positive number $K$ such that the function  $e^{-K z^2}\varphi(z)$ is bounded on the boundary of the rectangle $$P(k):=\{z=x+iy:\,\, \tilde x_k\leq x\leq x_k,\, y_1\leq y\leq y_2,\,\,\, k\in\N\}.$$
Hence,
$$|\varphi(z)|\leq Ce^{K x^2},\quad z=x+iy\in P(k),$$ where $C,K$ do not depend on $k.$ Letting  $k\to\infty$ and using the $\pi i/a$-periodicity of $\varphi$, we see that $\varphi$  satisfies
\eqref{eest}. 
\end{proof}
 
 We conclude this section with two technical remarks.
\begin{enumerate}
    \item[(1)] The arguments below do not depend on the value of the shape parameter $a > 0$ in the functions $H_a$ and $G_a$. For simplicity of presentation, we retain $a$ in the  statements, but set $a=1$ at the beginning of the proofs. We also set $G:=G_1$ and $H:=H_1$. 
    \item[(2)]Throughout the paper, the symbol $C$ denotes a positive constant whose value may change from line to line. 
\end{enumerate}

\section{Gaussian kernel: Proof of Theorem~\ref{th:gauss_ext}}

\subsection{Auxiliary lemmas}

Recall that every element $f\in V^\infty(G_a)$ is an entire function of the form
\begin{equation}\label{f}f(z)=\sum_{n\in\Z}c_ne^{-a(z-n)^2},\quad \{c_n\}\in l^\infty(\Z).\end{equation}This yields:
 \begin{equation}\label{order_2}
        |f(z)|\leq \sup_{n\in\Z}|c_n|\sum_{n\in\Z}\left|e^{-a(z-n)^2}\right|\leq Ce^{ay^2},\quad z=x+iy\in\mathbb{C}.
    \end{equation}

We start with the following characterization of the functions from the space $V^{\infty}(G_a).$ 
\begin{lemma}\label{gauss_char}
Let $a>0$. The following statements are equivalent.

    \item[$(i)$] $f\in V^\infty(G_a);$
    \item[$(ii)$] $f(z)=e^{-az^2}\varphi(z)$, where $\varphi(z)$ is a $\pi i/a$-periodic entire function satisfying  
\begin{equation}\label{ff}
    |\varphi(z)| \le C e^{ax^2}, \quad z = x + iy \in \Cc.
   \end{equation}

\end{lemma}

\begin{proof}
Assume $a=1.$

    (i) $\Rightarrow$ (ii). 
    Let $f\in V^\infty(G)$. Set
    $$
    \varphi(z): = e^{z^2}f(z)=e^{z^2} \sum \limits_{n \in \Z} c_n e^{-(z-n)^2} = \sum \limits_{n \in \Z} e^{2 z n} c_n e^{-n^2}.
    $$Clearly, $\varphi$ is $\pi i$-periodic, and by  \eqref{order_2} it satisfies \eqref{ff}.

   (ii) $\Rightarrow$ (i). As in the proof of Lemma \ref{c24}, we use the change of variables $w=e^{2z},$ and set $$\psi(w):=\varphi\left(\frac{\log w}{2}\right)=\sum_{k\in\Z}a_n w^n.$$    
By \eqref{ff} (with $a=1$), 
\begin{equation}\label{psi_e1}|\psi(w)|\leq C e^{(\log |w|)^2/4},\quad w\ne0. \end{equation}

Write $\psi=\psi_++\psi_-,$ where $\psi_\pm$ are defined in the proof  of Lemma \ref{c24}.
Note that $\psi_+$ satisfies estimate (\ref{psi_e1}) for all $|w|\geq 1$.
    Choosing $R=e^{2n}$ in Cauchy's inequality, we get
    $$|a_n|\leq \frac{M(R,\psi)}{R^n}\leq C\exp\left(\frac{\log^2 R}{4}-n\log R\right)=Ce^{-n^2}$$
     for all $n\geq0$.
    The same estimate holds for every $n<0.$ This means that the numbers $c_n:=a_ne^{n^2}$ are bounded, and so
the function
    $$f(z):=e^{-z^2}\psi(e^{2z})=\sum_{n=-\infty}^\infty c_ne^{-(z-n)^2}$$
    belongs to $V^\infty(G).$ This finishes the proof of the lemma.
\end{proof}

\begin{lemma}\label{lgf}
    Assume $f\in V^\infty(G_a)\cap L^1(\R)$. Then $f\in V^1(G_a).$
\end{lemma}

\begin{proof} We set $a=1.$

  Assume $f\in V^\infty(G)\cap L^1(\R)$. Fix any positive $\varepsilon<1$ and consider the function 
 $$f_\epsilon(x):=f\left(\frac{x}{1+\varepsilon}\right)e^{-\varepsilon x^2/(1+\varepsilon)^2}=\sum_{n\in\Z}c_n(f)e^{-(x-n)^2/(1+\varepsilon)-\varepsilon n^2/(1+\varepsilon)}.$$Clearly, $f_\epsilon\in V^1(G_{1/(1+\varepsilon)})$ and $\|f_\varepsilon\|_1\leq (1+\varepsilon)\|f\|_1\leq 2\|f\|_1.$ Note also that the Fourier transform of $f_\varepsilon$ is given by
 $$\hat f_\varepsilon(t)=\int_\R e^{-2\pi i tx}f_\varepsilon(x)\,dx=\sqrt{\pi(1+\varepsilon)}\left(\sum_{n\in\Z}c_n(f)e^{-\varepsilon n^2/(1+\varepsilon)}e^{-2\pi i n t}\right) e^{-\pi^2(1+\varepsilon)t^2}.$$

    Since the trigonometric system $\{e^{2\pi int}\}_{n\in\Z}$ is not complete in $L^2(I),$ for every interval $I$ of length $>1,$ there is a function $\Psi\in C^\infty(\R)$ which vanishes for $|t|\geq 2$ and satisfies 
    $$\int_\R e^{2\pi int}\Psi(t)\,dt=
    \begin{cases}
 \  1,& n=0;\\
 \ 0,& \ n\in\Z\setminus\{0\}.
\end{cases}
   $$

For every $\alpha>0$ we set $$\Psi_\alpha(t):=\sqrt{\frac{\alpha}{\pi}}e^{\pi^2 t^2/\alpha}\Psi(t),$$and denote by $\psi_\alpha$ its inverse Fourier transform. Clearly, $\Psi_\alpha\in C^\infty(\R)$, and   it vanishes for $|t|\geq 2 $, so that $\psi_\alpha$ is a Schwartz function for every $\alpha>0.$ From the Plancherel's theorem, we get
$$\int_\R e^{-\alpha(x-n)^2}\psi_\alpha(x-m)\,dx=\int_\R e^{-\alpha(x-n+m)^2}\psi_\alpha(x)\,dx=\int_\R e^{2\pi i(n-m)t}\Psi(t)\,dt= \begin{cases}
 \  1,& n=m;\\
 \ 0,&  n \neq m.
\end{cases}$$

We now assume that $\alpha=1/(1+\varepsilon)\in[1/2,1]$. We have
\begin{equation}\label{eq:ortck}
    \int_\R f_\varepsilon(x) \psi_{1/(1+\varepsilon)}(x-k) \, dx=c_k(f)e^{-\varepsilon n^2/(1+\varepsilon)}.
\end{equation}

Clearly, 
$$\|\psi_\alpha\|_{L^\infty(\R)}+ \| x^2\psi_\alpha(x)\|_{L^\infty(\R)}\leq \|\Psi_\alpha\|_{L^1(\R)}+ \|(\Psi_\alpha)''\|_{L^1(\R)}<C,$$where $C$ does not depend on $\varepsilon.$
Therefore, 
 \begin{equation}\label{epsi}\sup_{x\in\R}\sum_{n\in\Z}|\psi_{\alpha}(x-n)| <C,\quad \text{for every } \alpha=1/(1+\varepsilon)\in [1/2,1].\end{equation}

Consider the integral $$I_{\varepsilon}:=\int_\R f_\varepsilon(x)\sum_{k\in\Z}\frac{\overline{c_k(f)}}{|c_k(f)|}\psi_{1/(1+\varepsilon)}(x-k)\,dx.$$By \eqref{epsi}, 
\begin{equation}\label{e0e}|I_\varepsilon|<C\|f_\varepsilon\|_1\leq 2C\|f\|_1,\quad 0<\varepsilon\leq 1.\end{equation} On the other hand, since  the sequence $\{c_n(f)\exp(-\varepsilon n^2/(1+\varepsilon))\}\in l^1(\Z)$, by the right-hand-side inequality in \eqref{eq:norm_equiv}, the function  $$\sum_{n\in\Z}|c_n(f)|e^{-\varepsilon n^2/(1+\varepsilon)}e^{-(x-n)^2/(1+\varepsilon)}$$belongs to $L^1(\R).$ By the Dominated convergence theorem, integrating term for term  and using \eqref{eq:ortck}, we obtain $$I_\varepsilon =\sum_{n\in\Z}|c_n(f)|e^{-\varepsilon n^2/(1+\varepsilon)}.$$ Using \eqref{e0e} and  letting $\varepsilon \to 0$, we get $\{c_n(f)\}\in l^1(\Z),$ so that $f\in V^1(G).$
\end{proof}

\begin{lemma}\label{lcl}Let $f\in V^1(G_a)$. The following conditions are equivalent:

{\rm(i)} $e^{2ax}f(x)\in L^1(\R);$

{\rm(ii)} $\{e^{2an}c_n(f)\}_{n\in\Z}\in l^1(\Z);$ 

{\rm(iii)} $e^{2ax}f(x)\in V^1(G_a)$.

 A similar result holds if we set "$-$" in every exponent in the three conditions above.
   
\end{lemma}

\begin{proof}

For brevity, set $a=1$.

(i)$\Rightarrow$(iii) and  (ii).
Assume (i) holds true. Then $e^{bx}f(x)\in L^1(\R),$ for every $0\leq b\leq 2,$ and so the Fourier transform
$$\hat f(z)=\int_\R e^{-2\pi izx}f(x)\,dx=\frac{1}{\sqrt\pi}\left(\sum_{n\in\Z}c_n(f)e^{-2\pi izn}\right)e^{-\pi^2 z^2}
$$admits an analytic continuation into the strip $0< {\rm Im\,}\,z<\pi^{-1}$, and is continuous up to its boundary. This means that the $1$-periodic series
$$\sum_{n\in\Z}c_n(f)e^{-2\pi izn}$$ also admits an analytic continuation into the strip $0< {\rm Im\,}\,z<\pi^{-1}$, and is continuous up to its boundary. 
This property is obviously true for the series $$\sum_{n\geq0}c_n(f)e^{-2\pi izn}.$$This is equivalent to the condition that the power series
$\sum_{n\geq0}c_nw^n$ admits an analytic continuation to the disk $|w|<e^2,$ and is bounded above by some constant $C$ in its closure. By  Cauchy's inequality,
$$|c_n|\leq \frac{C}{e^{2n}},\quad n\geq0,$$which proves that $\{c_n e^{2n}\}_{n\geq0}\in l^\infty(\mathbb N).$ 

Observe that 
$$e^{2x}f(x)=e^{2x}\sum_{n\in\Z}c_n(f)e^{-(x-n)^2}=\sum_{n\in\Z}c_n(f)e^{2n+1}e^{-(x-n-1)^2}=:g(x),$$and set $c_n(g):=c_n(f)e^{2n+1}.$
Then $g\in V^\infty(G).$ 
By (i) and Lemma \ref{lgf}, $g\in V^1(G)$, which proves (iii) and (ii).

(ii)$\Rightarrow$(i). Let $g$ be the function above.
By (ii), $\{c_n(g)\}\in l^1(\Z)$ and since the Gaussian generator satisfies condition \eqref{eq:norm_equiv},  condition (i) holds true.

\end{proof}

\subsection{Proof of Theorem~\ref{th:gauss_ext} {\rm (i)}}\label{sec1}

Again, for simplicity, we set $a=1$. 

{\bf Necessity.}  
Assume a  function $f\in V^1(G)$  vanishes at $\lambda$ and $\bar\lambda$, where Im$\,\lambda\in(0,\pi)$.

 Assume $\, {\rm Im\,}\lambda\ne\pi/2.$ 
Set
$$
\tau(z) := \frac{1}{(e^{2z} - e^{2\lambda})(e^{2z} - e^{2 \bar\lambda})}, \quad z \in \Cc.
$$
Note that $\tau(x)$ is a real non-constant function bounded on $\R$.
Since the function $e^{z^2}f(z)$ is $\pi i$-periodic, we see that  $f$ vanishes on the set $\{\lambda,\bar\lambda\}+\pi i \Z.$ Therefore, $\tau f$ is an entire function.

By Lemma~\ref{h_lemma}, to prove that $f$ is not an extreme point, it suffices to show that $q:=\tau f$ is an element of $V^1(G).$
 Indeed,  since $f$ satisfies \eqref{order_2} (with $a=1$), then so does $q$. It is clear that the function $\varphi(z):=e^{z^2}q(z)$ is $\pi i$-periodic, 
satisfies \eqref{ff} and $q$ belongs to $L^1(\R)$. By  Lemmas \ref{gauss_char} and \ref{lgf}, we conclude that  $q\in V^1(G)$. 
Therefore, $f$ is not extreme.

 Assume ${\rm Im \,}\lambda=\pi/2$. Then the function $$\tau(z):=\frac{1}{e^{2z}-e^{2\lambda}}$$ is real and bounded on the real line. The  argument  in part (i) leads to the conclusion that $f$ is not extreme. 

{\bf Sufficiency.}
Assume that there is no $\lambda \in \Cc \setminus\{\R + \pi i k\}_{k \in \Z}$ such that $f(\lambda) = f(\bar \lambda) = 0.$ By Lemma~\ref{h_lemma}, to show that $f$ is extreme, it suffices to prove that every function $$\tau(x) := \frac{q(x)}{f(x)}, \quad q \in V^1(G),$$ which is real and bounded on $\R$ is constant.

Since $$\tau(x)=\frac{e^{x^2} q(x)}{e^{x^2}f(x)}:=\frac{\varphi_1(x)}{\varphi_2(x)},$$where $\varphi_j$  are $\pi i$-periodic entire functions, then $\tau$ is also $\pi i$-periodic meromorphic function.

 Assume  $\tau$ has a pole at some point $\lambda\in\mathbb C$. Since $\tau$ is real, it also has a pole at $\bar\lambda$. Since it is bounded on $\R$ and $\pi i$-periodic, we may assume that  $0<\,$Im$\,\lambda<\pi.$ 
However, this implies 
 that $f$ vanishes at both $\lambda$ and $\bar\lambda$, which is not the case. We conclude that $\tau$ is an entire function. By Lemma \ref{c24}, $\tau$ satisfies 
\eqref{eest}. 
Recall that $\tau$ is  bounded on $\R$, and so also on the horizontal lines ${\rm Im\,} z=\pi k, k\in\Z$. By the Phragmen--Lindel\"{o}f principle, it is a constant function.
This finishes the proof of the theorem on extreme points for the Gaussian kernel.

\subsection{Proof of Theorem~\ref{th:gauss_ext}  {\rm (ii)}}\label{34}

We set $a=1.$

{\bf Necessity.}
    We start with proving that condition \eqref{eq:exposed_cond} is necessary. 
    Assume that it  is not satisfied: there is a point $\lambda \in \R$ such that  $f(\lambda)=f'(\lambda)=0$.

    Set 
    $$
    h(z):=\frac{1}{(e^{2z}- e^{2\lambda})^2}.
    $$
    Since the function $e^{z^2}f(z)$ is $\pi i$-periodic, both $f$ and $f'$ vanish on the set $\lambda+\pi i\Z$, and so   $fh$ is an entire function. Similarly  to the proof in Section \ref{sec1}, one can check $fh\in V^1(G)$.
    Since $h$ is non-negative on $\R$, by Lemma~\ref{exp_lemma}, we conclude that $f$ is not exposed. 

It is an immediate consequence of 
Lemma \ref{lcl} that condition \eqref{ex} is necessary.

{\bf Sufficiency.}
    Assume that $f$ satisfies conditions~\eqref{eq:exposed_cond} and \eqref{ex}. To prove that $f\in \expo(V^1(G))$, by Lemma~\ref{exp_lemma}, it suffices to show that every non-negative function $\tau$ that can be represented as $\tau = q/f$ for some $q \in V^1(G)$ is constant.
    
  Let us show that $\tau$ is an entire function. Similarly to the proof in Section \ref{sec1}, one may show that $\tau$ does not have poles in the strip $0<\,$Im$\,z<\pi$.
  Therefore, due to  the $\pi i$-periodicity of $\tau$, if it has a pole, then it has a pole at some $\lambda\in\R$. By \eqref{eq:exposed_cond}, every such pole must have order one.
Then $$\tau(x-\lambda)= \frac{q(x-\lambda)}{f'(\lambda)(x-\lambda)}+O(1), \text{ as } x\to \lambda,$$which clearly implies that $\tau$ cannot be non-negative on $\R$. Thus, $\tau$ is an entire function.
 By Lemma  \ref{c24}, $\tau$ satisfies \eqref{eest}.

Write
$$\tau(z)=\sum_{n\in\Z}a_ne^{2nz}=a_0+\sum_{n>0}a_ne^{2nz}+\sum_{n<0}a_ne^{2nz}:=a_0+\tau^+(z)+\tau^-(z).$$We have to check that $\tau^-=\tau^+=0.$ 

Assume  $\tau^+$ contains a finite number of non-zero coefficients, and set $n_0:=\max_{n>0}\{a_n\ne0\}$. Then condition $\tau f\in V^1(G)$ implies $e^{2n_0 x}f(x)\in L^1(\R),$ and so $e^{2x}f(x)\in L^1(\R)$ which contradicts to condition~\eqref{ex}. 

Assume  $\tau^+$ contains an infinite number of non-zero coefficients. Then the function $$\psi^+(w):=\tau^+(\log w/2)=\sum_{n>0}a_nw^n$$ is an entire transcendental function. A direct consequence of the Cauchy estimate
$$|a_n|\leq\frac{M(R,\psi^+)}{R^n},\quad \text{for every } R>0,$$ is that $$M(R,\psi^+)R^{-n}\to\infty,\ \ R\to\infty,\quad \text{for every }n>0.$$
In particular, $M(r,\psi^+)\geq Cr^8, r\geq 1$. We conclude that $$\sup_{0\leq x\leq r}|\tau^+(x+iy)|\geq Ce^{8x},\quad x\geq 0.$$

Recall that $\tau^+$ is an entire $\pi i$-periodic function satisfying \eqref{eest}.
Similarly to the proof  of Lemma \ref{c24}, by applying Lemma~\ref{hayman} to the entire function of slow growth $\psi^+$ defined above, one may show that there is a sequence $r_k\to\infty$ such that 
\begin{equation}\label{eq:tau_plus_global}
  \sup_{0\leq x\leq r,y\in\R}|\tau^{+}(x+iy)| > Ce^{4r}, \quad r=r_k,  
\end{equation}
and that there exist $y_1 > 0, y_2 < 0 $, and $R > 0$ such that 
\begin{equation}\label{eq:tau_rays}
  \sup_{0\leq x\leq r}|\tau^{+}(x+iy_j)| > Ce^{4r}, \quad r>R, \,\, j=1,2.  
\end{equation}Therefore, $|\tau^+(x+iy)|\geq Ce^{4x}$ on the boundary of each rectangle $$P_k:=\{z=x+iy: 0\leq x\leq r_k, y_2\leq y\leq y_1\}, \,\, k\in\N.$$

Recall that $\tau f=g\in V^1(G)$. Hence, $\tau^+ f=g-\tau^-f$. By \eqref{order_2}, $f$ and $g$  are uniformly bounded above on $P_k, k\in\N.$ Clearly, the same is true for $\tau^-$. We conclude that $|\tau^+(z)f(z)|<C, z\in P_k,$ where $C$ does not depend on $k$.
By \eqref{eq:tau_plus_global} and \eqref{eq:tau_rays}, the function $e^{4z}f(z)$ is bounded above on the boundary of $P_k$. From \eqref{order_2} we see that this function has order two (or less), so the  Phragmen-Lindel\"of principle proves that it is bounded in $P_k$. In particular, we have $|e^{4x}f(x)|\le C, x \ge 0$, which contradicts \eqref{ex}. We conclude that the series for $\psi$ cannot contain an infinite number of terms. This and the previous argument prove that $\tau^+=0.$ Similarly, one may check that $\tau^-=0$, so that $\tau=const.$

\section{Hyperbolic secant: Proof of Theorem~\ref{th:hyp_ext}}

\subsection{Auxiliary results}

Recall that every $f\in V^1_\Gamma(H_a)$ admits a representation 
\begin{equation}\label{eq:f_sec}
  f(x)=\sum_{\gamma\in\Gamma}\frac{c_\gamma}{e^{a(x-\gamma)}+e^{a(\gamma-x)}},\quad c_\gamma\in l^1(\Gamma),  
\end{equation}
where $a>0$ and $\Gamma \subset \R$ is a separated set. One may check that $f$ admits an extension to a meromorphic function, which may only have simple poles on the set 
 \begin{equation}\label{pa}  \mathcal{P}_a:= \Gamma+\frac{i\pi}{a}\left(\Z+\frac{1}{2}\right) \end{equation} and   
 \begin{equation}\label{res}c_\gamma(f)=2ai(-1)^k\res\left(f,\gamma+ \frac{\pi i k}{a}+\frac{\pi i}{2a}\right),\quad  k\in\Z,\, \gamma \in \Gamma. \end{equation}

Clearly, every $f\in V^1_{\Gamma}(H_a)$ satisfies the following four conditions: 
\begin{equation}\label{eq:pi_odd}
    f(z+\pi i/a) = -f(z), \quad \text{ for every \,} z \in \Cc,
\end{equation}
 \begin{equation}\label{1} f \text{ is a meromorphic function, whose every pole is simple and lies in } \mathcal{P}_a,  \end{equation} 
 \begin{equation}\label{2}
 f\text{ is bounded  above on every set }
    \mathbb C\setminus \left(\mathcal{P}_a+\{z:|z|<\delta\}\right) ,\quad \forall\delta>0,
 \end{equation} 
\begin{equation}\label{3}
\{\res(f,\gamma+ \pi i/(2a)\}_{\gamma \in \Gamma}\in l^1(\Gamma).
 \end{equation}

 The last condition follows easily from \eqref{res}.
 
\begin{lemma}\label{lem:secant_represent}
Let $a>0$. The following statements are equivalent:

{\rm(i)}  $f\in V^1_\Gamma(H_a)${\rm ;} 

{\rm(ii)} $f\in L^1(\R) $ and is a meromorphic function satisfying  \eqref{eq:pi_odd}--\eqref{3}.
\end{lemma}
\begin{proof}
As usual, set $a=1.$
The implication (i) $\Rightarrow$ (ii) is already explained.

   (ii) $\Rightarrow$ (i). Set
   $$g(z):=\sum_{\gamma\in\Gamma}\frac{c_\gamma}{e^{(z-\gamma)}+e^{(\gamma-z)}},$$where the coefficients $c_\gamma$ are defined in \eqref{res}. By  \eqref{3}, $g\in V^1_{\Gamma}(H)$. 
   
   Consider the difference $q:=f-g$. It is clear that $q$ satisfies \eqref{2}.  By \eqref{eq:pi_odd} and \eqref{1}, one may easily check that it has a removable singularity at every point $z\in \mathcal{P}_1$.  Hence, $q$ is a bounded entire function, and so $q\equiv const$. Since both $f$ and $g$ are elements of  $L^1(\R)$, we conclude that $q\equiv0.$
\end{proof}

\subsection{Proof of Theorem~\ref{th:hyp_ext} (i)}\label{s42}
We set $a=1$.

{\bf Necessity.}
Suppose a function $f\in V^1_{\Gamma}(H)$ vanishes at the points
$\lambda, \bar \lambda,$ where $0<{\rm Im}\,\lambda <\pi$. We need to show that $f \notin \ext(V^1_{\Gamma}(H)).$ We will distinguish between two cases.

{\bf Case} $(i).$ Assume that ${\rm Im\,}\lambda\ne\pi/2$.
By Lemma~\ref{h_lemma}, we are done if we verify that the function 
\begin{equation}\label{q}q(x):=\frac{f(x)}{(e^{2x}-e^{2\lambda})(e^{2x}-e^{2\bar \lambda}) }\end{equation}
belongs to $V^1_{\Gamma}(H)$.
It suffices to check that $q$ satisfies condition (ii) of Lemma~\ref{lem:secant_represent}.

It is straightforward to check that $q$ belongs to $ L^1(\R)$ and satisfies \eqref{eq:pi_odd} and \eqref{3}. 

Denote by $\Lambda\subset\mathbb C$ the zero set of the function $h(z):= (e^{2z}-e^{2\lambda})(e^{2z}-e^{2\bar\lambda})$.
The zeros of $f$ are $\pi i$-periodic, so that $f$ vanishes on $\Lambda$. Hence, $q$ satisfies \eqref{1}.

It is clear that $h$ is bounded below outside of any set $S_\delta:=\Lambda+\{|z|\leq\delta\},\delta>0$.
We assume that $\delta$ is so small that $S_\delta$ does not intersect the set of poles of $f$. Since $f$ satisfies \eqref{2}, it follows that $|q(z)|$  is uniformly bounded above on the set of points $z$ that lie at the distance $\geq\delta$ to both $\Lambda$ and the set of poles of $f$. However, since $q$ is analytic on  $S_\delta$,  by the maximum modulus principle,  it is also bounded above on $S_\delta$. Hence $q$ satisfies \eqref{2}, and so $q\in V^1_{\Gamma}(H)$.

{\bf Case }$(ii).$ Assume ${\rm Im \,}\lambda=\pi/2$. Then the function $\left(e^{2z}-e^{2\lambda}\right)^{-1}$ is real and bounded on the real line. The argument in part (i) leads to the conclusion that $f$ is not extreme. 

To finish the proof of the necessity part, it suffices to show that if $c_{\sigma} = 0$ for some $\sigma \in \Gamma$ then $f$ is not extreme. Take $\sigma' \in \Gamma$ such that $c_{\sigma'} \neq 0$ and set
$$
\tau(x) = \frac{e^{(x - \sigma')} + e^{(\sigma' - x)}}{ e^{(x-\sigma)} + e^{(\sigma -x)}}.
$$
It is clear that $\tau$ is real and bounded on $\R$. 
Using the same approach, one can verify that $g:=\tau f\in L^1(\R)$ and  satisfies \eqref{eq:pi_odd}--\eqref{3}.
Thus, Lemma~\ref{lem:secant_represent} asserts that $g \in V^1_{\Gamma}(H).$

{\bf Sufficiency.} Assume that $f$ satisfies conditions~\eqref{s_zeros} and~\eqref{eq:c_non_zero}.
By Lemma~\ref{h_lemma}, it suffices to show that every real-valued function $\tau$  on $\R$  that is bounded and satisfies $\tau f=:q\in V^1_{\Gamma}(H)$ must be constant.

\begin{claim}
    $\tau$ is an entire function.
\end{claim}
Indeed, the poles of $\tau$ may only arise from the poles of $q$ or from the zeros of $f$. All poles of $q$ are simple and lie on the set $\mathcal{P}_1$ defined in \eqref{pa}. By~\eqref{eq:c_non_zero}, the function $1/f$ has zero at every point of this set,  so $\tau$ cannot have poles there. Moreover, $\tau$ has no poles on  $\R$, since it is bounded on $\R$. Finally, since $\tau$ is $\pi i$-periodic and real-valued on $\R$, the presence of any pole would force the existence of poles at some $\lambda$ and its conjugate $\bar\lambda$, with $0<\,$Im$\,\lambda<\pi$, which contradicts~\eqref{s_zeros}.

\begin{claim}
    $\tau$ satisfies the estimate in \eqref{eest}.
\end{claim}
Set $$P(z)=P_+(z)\cdot P^-(z):=\prod_{\gamma\in\Gamma,\gamma\geq0}\left(1-e^{2(z-\gamma)}\right)\cdot \prod_{\gamma\in\Gamma,\gamma<0}\left(1-e^{-2(z-\gamma)}\right).$$

Clearly, $P^+(z)$ is bounded for $z=x+iy, x<0.$ For $x\geq0$, we have 
$$|P^+(z)|\leq \prod_{\gamma\geq0}\left(1+e^{2(x-\gamma)}\right)\leq\prod_{0\leq \gamma\leq 2x}\left(1+e^{2x}\right)\cdot\prod_{\gamma> 2x}\left(1+e^{-\gamma}\right). $$
Since the set $\Gamma$ is separated, there is a constant $\rho$ such that $\#(\Gamma\cap [0,x])\leq \rho x, x\geq 1.$
Hence $P^+$ satisfies \eqref{eest}, and the same estimate applies to $P^-$. We conclude that $P$ satisfies \eqref{eest}.

Clearly, the functions $qP$ and $fP$ are entire and satisfy \eqref{eq:pi_odd}.
By Lemma \ref{lem:secant_represent}, both $q$ and $f$ are bounded  above outside the set $\mathcal{P}_1$. 
Hence, by the maximum modulus principle, both  $qP$ and $fP$ satisfy \eqref{eest}. 

At this point, applying the same argument used at the end of the proof of sufficiency in Theorem~\ref{th:gauss_ext}, part $(i)$, we conclude that $\tau$ is necessarily constant.

\subsection{Proof of Theorem~\ref{th:hyp_ext} (ii)}  
As usual, set $a=1.$

{\bf Necessity.}
    Assume that condition \eqref{eq:exposed_cond} is not satisfied: 
    there is a $\lambda \in \R$ such that $f$ and $f'$ vanish at $\lambda$.     
    By Lemma~\ref{exp_lemma}, we are done if we verify that the function 
$$q(z):=\frac{f(z)}{(e^{2z}-e^{2\lambda})^2}$$
belongs to $V^1_{\Gamma}(H)$.  Since $f$ satisfies \eqref{eq:pi_odd}, then $q$ is a meromorphic function with simple poles on $\mathcal{P}_1$. 
Similarly to the proof in Section \ref{s42}, one can check that $q$ satisfies the assumptions of Lemma~\ref{lem:secant_represent} (ii), and so $q\in V^1_{\Gamma}(H).$ Hence, condition \eqref{eq:exposed_cond} is necessary.

Let us now check that conditions $e^{\pm 2x}f(x)\not\in L^1(\R)$ are also necessary. 
It suffices to show that the conditions $e^{\pm 2x}f(x)\in L^1(\R)$ imply $e^{\pm 2x}f(x)\in V^1_\Gamma(H)$. We will only prove this for the condition  $e^{2x}f(x)\in L^1(\R)$, since the proof for the remaining case is similar.

Observe that the function $e^{2z}f(z)$ satisfies \eqref{eq:pi_odd}, so that it is also integrable on the line $\R+i\pi$. 

In what follows, we construct an auxiliary function with prescribed values on the set $\Gamma + \pi i/2$. 
To this end, we use the solution of the interpolation problem in the Bernstein space $B(\sigma)$. 
This space consists of all entire functions of exponential type at most $\sigma$ that are uniformly bounded on $\mathbb{R}$. In particular, every function from $B(\sigma)$ is uniformly bounded on every horizontal strip in $\Cc$ (see, e.g., \cite{MR3468930} for further details).
Since the set $\Gamma$ is separated, the set $\Gamma+\pi i/2$ is an interpolation set for $B(\sigma)$, for every sufficiently large $\sigma>0$, see \cite{MR1057613} or \cite{MR3468930}. Hence, for every $\delta>0$ there is a function $B_\delta\in B(\sigma)$  such that
$$B_\delta(\gamma+i\pi/2)=-e^{i\pi\gamma\delta}\frac{\overline{c_\gamma(f)}}{|c_\gamma(f)|},\quad \text{for every } \gamma\in\Gamma.$$
Observe that the $L^\infty$-norm of $B_\delta$ on the strip $0\leq\, {\rm Im\,}z\leq\pi$ depends only on $\sigma$, the separation constant of $\Gamma$ and the sup-norm of the sequence $|-ie^{i\gamma\delta}c_\gamma/|c_\gamma||=1,\gamma\in\Gamma$, so that it does not depend on $\delta.$

Given $R>0$, consider the contour $C_R$ which consists of the segments $[-R,R]\cup[R,R+\pi i]\cup[R+\pi i, -R+\pi i]\cup[-R+\pi i,-R]$. Fix $\sigma>0$. By  Lemma 4.1, for every small $\epsilon>0$ and large $R$, such that $|R-\gamma|>\epsilon,\gamma\in\Gamma$, the function $f_\delta(z):=e^{2z}f(z)B(z)e^{-\delta z^2}$ is integrable on $C_R$, and its $L^1$-norm on $C_R$ is uniformly bounded as $R\to\infty$. Integrating $f_\delta$ over $C_R$ and letting $R\to\infty $, by the residue theorem, we have 
$$\lim\limits_{R \to \infty}\int_{C_R}f_\delta(z)\,dz=2\pi i \sum_{\gamma\in\Gamma}\text{Res}(f_\delta,\gamma+i\pi/2)=\pi e^{\delta\pi^2/4}\sum_{\gamma\in\Gamma}|c_\gamma|e^{2\gamma-\delta \gamma^2}<K,$$where the constant $K$ does not depend on $\delta.$ By letting $\delta\to 0$, we conclude that the series converges:
\begin{equation}\label{cgamma}\sum_{\gamma\in\Gamma}|c_\gamma|e^{2\gamma}<\infty.\end{equation}

Now, write $$e^{2x}f(x)=e^{2x}\sum_{\gamma\in\Gamma}\frac{c_\gamma}{e^{x-\gamma}+e^{\gamma-x}}=\sum_{\gamma\in\Gamma}c_\gamma\left(e^\gamma e^x-\frac{e^{2\gamma}}{e^{x-\gamma}+e^{\gamma-x}}\right).$$
By \eqref{cgamma}, we see that the function $$\sum_{\gamma\in\Gamma}\frac{c_\gamma e^{2\gamma}}{e^{x-\gamma}+e^{\gamma-x}}\in L^1(\R).$$Therefore,
$$e^{x}\sum_{\gamma\in\Gamma}c_\gamma e^{\gamma}\in L^1(\R),$$which implies \begin{equation}\label{cgamma1}\sum_{\gamma\in\Gamma}c_\gamma e^{\gamma}=0.\end{equation}We conclude that 
$$e^{2x}f(x)=\sum_{\gamma\in\Gamma}\frac{c_\gamma e^{2\gamma}}{e^{x-\gamma}+e^{\gamma-x}}\in V^1_\Gamma(H).$$

\begin{corollary}\label{c43}
  Let $f\in V^1_\Gamma(H)$. The function  $e^{2x}f(x)$ is integrable if and only if both conditions \eqref{cgamma} and \eqref{cgamma1} are fulfilled.  
\end{corollary}
 Similarly, one may check that   $e^{-2x}f(x)\in L^1(\R)$ if and only if we have
 $$\sum_{\gamma\in\Gamma}|c_\gamma|e^{-2\gamma}<\infty,\quad \sum_{\gamma\in\Gamma}c_\gamma e^{-\gamma}=0.$$

{\bf Sufficiency.}
The proof follows the same strategy as the sufficiency part of Theorem~\ref{th:gauss_ext}\,(ii), see Section \ref{34}. For brevity, we leave the remaining details to the reader.

\section{Acknowledgements}
The authors are very thankful to Aleksei Kulikov for several helpful remarks that improved the manuscript.

\end{document}